\documentclass[12pt]{amsart}

\def\IR{{\Bbb R}}

\def\IC{\Bbb C}
\def\ID{{\Bbb D}}
\def\IT{{\Bbb T}}

\def\zbar{{\overline{z}}}

\newcommand{\ttt}{\mathbf T}

\usepackage[psamsfonts]{amssymb}
\usepackage{amsfonts,amsmath}
\usepackage{amsthm}
\usepackage{graphicx}

\newcounter{minutes}
\divide\time by 60
\newcounter{hours}
\multiply\time by 60 \addtocounter{minutes}{-\time}

\newtheorem{theorem}{Theorem}
\newtheorem{lemma}{Lemma}
\newtheorem{corollary}{Corollary}

\newtheorem{proposition}{Proposition}[section]
\newtheorem{definition}{Definition}[section]
\newtheorem{remark}{Remark}[section]

\usepackage{amsmath, amssymb, amscd, amsxtra, txfonts}

\title[{Quasiconformal maps with controlled Laplacian}]{Quasiconformal maps with controlled Laplacian}

\begin{document}

\author[Kalaj]{David Kalaj}
\address{University of Montenegro, Faculty of Natural Sciences and
Mathematics, Cetinjski put b.b. 81000 Podgorica, Montenegro}
\email{davidkalaj@gmail.com}
\author[Saksman]{Eero Saksman}
\address{Department of Mathematics and Statistics, University of
Helsinki, PO~Box~68, FI-00014 Helsinki, Finland}
\email{eero.saksman@helsinki.fi}
\date{November 02, 2014}
\thanks{E.S.\  was
supported by the Finnish CoE in Analysis and Dynamics Research,
and by the Academy of Finland, projects
113826 \& 118765}


\maketitle

\makeatletter\def\thefootnote{\@arabic\c@footnote}\makeatother

\begin{abstract}
We establish that every $K$-quasiconformal mapping $w$ of the unit disk $\ID$ onto a $C^2$-Jordan domain $\Omega$ is Lipschitz provided that $\Delta w\in  L^p(\ID)$ for some $p>2$. We also prove that if in this situation    $K\to 1$ with $\|\Delta w\|_{L^p(\ID)}\to 0$, and $\Omega \to \ID$ in  $C^{1,\alpha}$-sense with $\alpha>1/2,$ then the bound for  the Lipschitz constant tends to $1$. In addition, we provide a quasiconformal  analogue of the Smirnov theorem on absolute continuity over the boundary.
\end{abstract}

\maketitle

\section{Introduction}\label{intsec}

Recall that the map $w:\ID\to \IC$ of the unit disc to the complex plane is quasiconformal  if it is a sense preserving homeomorphism that has locally $L^2$-integrable weak partial derivatives, and it satisfies for almost every $z\in \ID$ the distortion inequality $|w_\zbar| \leq k |w_z|$, where $k<1.$ In this situation we say that $w$ is $K$-quasiconformal, with $K:=(1+k)/(1-k).$ We refer to  \cite{Ahl} and \cite{AIMb} for basic  notions and results of the quasiconformal theory.
Quasiconformal self-maps of the disc, even when locally $C^2$-smooth inside $\ID$, need not to be Lipschitz. However,
in the situation where $w:\ID\to \ID$ is a quasiconformal  homeomorphism that is also harmonic, Pavlovi\' c \cite{MP} proved  that $f$ is bi-Lipschitz.
Many generalisations of this result for harmonic maps heve been proven since, we refer e.g. to \cite{mana} and \cite{kave} and the references therein.

Our paper addresses the following  problem: how much one can relax the condition of harmonicity of  the quasiconformal map  $w$, while  still  being able to deduce  the Lipschitz property of $w$ --  in this situation it is less natural to inquire $w$  to be bi-Lipschitz. Better insight to this kind of questions  ought to be useful also in applications to non-linear elasticity.
A natural measure for the deviation from harmonic functions is to consider $\|\Delta w\|_{L^p(\ID)}$  for some $p\geq 1$ and ask whether finiteness of this quantity enables one to make the desired conclusion.
Our first  result yields the following:
\begin{theorem}\label{th:first} Assume that  $g\in L^p(\ID)$ and  $p>2$.
If $w$ is a $K$-quasiconformal  solution of  $\Delta w= g,$  that maps the unit disk onto a bounded Jordan domain $\Omega\subset\IC$ with $C^{2}$-boundary, then $w$ is Lipchitz continuous. The result is sharp since it fails in general if $p=2.$
\end{theorem}
\noindent The proof is given in Section \ref{se:proof1}.

Our second  result shows that, in the setting of Theorem 1, the Lipschitz constant of a normalised map $f$ becomes arbitrarily close to 1 if the image domain $\Omega$ approaches the  unit disc in a suitably defined $C^{1,\alpha}$-sense, and  if  deviations both from conformality and harmonicity tend to zero. Below we identify $[0,2\pi)$ and the boundary of the unit disc $\IT$ in the usual way.
\begin{theorem}\label{th:second}
  Let $p>2$ and   assume that $w_n:\ID\to\Omega_n $ is a $K_n$-quasi-conformal normalised map normalised by $w(0)=0$, and with
\[
\lim_{n\to\infty} K_n=1 \quad\textrm{and}\quad  \lim_{n\to\infty} \|\Delta w_n\|_{L^p(\ID)}=0.
\]
Moreover,  assume that for each $n\geq 1$  the bounded Jordan domain $\Omega_n$ approaches the unit disc in the $C^{1,\alpha}$-bounded sense.
More precisely, this means that there is a parametrisation
\[
\partial \Omega_n=\{ f_n(\theta) \; | \; \theta \in \mathbf{T}\},
\]
where  $f_n$ satisfies for some $\alpha >1/2$
\[
\| f_n(\theta)-e^{i\theta}\|_{L^\infty(\mathbf{T})}\to 0 \quad \textrm{as}\quad n\to\infty\quad \textrm{and}\quad  \sup_{n\geq 1}\| f_n(\theta)\|_{C^{1,\alpha}(\IT)}<\infty.
\]
Then for large enough $n$ the function $w_n$ is Lipschitz, and moreover its Lipschitz constant tends to 1 as $n\to\infty:$
\begin{equation}\label{eq:limit}
\lim_{n\to\infty} \| \nabla w\|_{L^\infty(\ID)} =1.
\end{equation}
\end{theorem}
\noindent This result will be obtained  as a corollary of slightly more general results in Section \ref{se:proof2} below. Together, our Theorems  \ref{th:first}  and \ref{th:second} considerably  improve the main result of the first author and Pavlovi\'c from \cite{trans}, where it was instead assumed that $\Delta w\in C(\overline{\ID}).$ Other related results are contained in \cite{JAM}, we refer to \cite{aips} and references therein for other type of connections between quasiconformal and Lipschitz maps.

In order to state our last theorem, we recall the  result of V. I. Smirnov, stating that a conformal mapping of the unit disk $\mathbf{U}$ onto a Jordan domain $\Omega$ with rectifiable boundary has a absolutely continuous extension to the boundary. This implies in particular that if  $E\subset \mathbf{T}$ is a set of zero 1-dimensional Hausdorff measure then its image $f(E)$ is a set of zero 1-dimensional Hausdorff measure  in $\partial\Omega$. Further, this result has been generalized for the class of q.c. harmonic mapping by several authors (see e.g. \cite{pk,kmm}). On the other hand if we assume that $f$ is merely quasiconformal, then its boundary function need not be in general an absolutely continuous function. In Section \ref{se:proof3} we prove the following generalization of  Smirnov's theorem for quasiconformal mappings,  subject again to an size condition on their Laplacian:

\begin{theorem}\label{th:smirnov}
Let $f$ be a quasiconformal mapping of the unit disk onto a Jordan domain with rectifiable boundary. Assume that  $\Delta f$ is locally integrable and satisfies
\[
|\Delta f(z)|\leq C(1-|z|)^{-a}
\]
for some constants $a<1$, and $C<\infty$. Then  $f_{|\ttt}$ is an absolutely continuous function. The result is optimal: there is a quasiconformal self-map $f:\ID\to\ID,$ with non-absolutely continuous boundary values, and such that $f\in C^\infty (\ID)$ and with $|\Delta f(z)|\leq C(1-|z|)^{-1}$
in $\ID.$
\end{theorem}

\noindent Another variant of the proof of the previous theorem goes as follows:
\begin{corollary}\label{co:smirnov}
If $f$ is a quasiconformal mapping of the unit disk onto a Jordan domain with rectifiable boundary such that $\Delta f\in L^{p}(\ID)$ for some $p>1$,  then $f_{|\ttt}$ is an absolutely continuous function. The claim fails in general if $p<1.$
\end{corollary}

 Further comments, generalizations and open questions related to the above  results are included in in Sections \ref{se:proof1}--\ref{se:proof3}.

\section{ Proof of Theorem \ref{th:first}: Lipschitz-property of qc-solutions to $\Delta f=g$}\label{se:proof1}

In what follows, we say that a bounded Jordan domain $\Omega\subset\IC$  has $C^2$-boundary if it is the image of the unit disc $\ID$ under a $C^2$-diffeomorphism of the whole complex plane onto itself. For planar Jordan domains this is well-known to be equivalent to the more standard definition, that requires the boundary to be locally isometric to the graph of a $C^2$-function on $\IR$. In what follows, $\Delta $ refers to the distributional Laplacian. We shall make use of the following well-known fact, whose proof we recall:
\begin{lemma}\label{le:interior_gradient}
 Assume that $w\in C(\overline{\ID})$ is such that $\|\Delta w\|_{L^p(\ID)}<\infty$ with $p>1$.

 \smallskip

 \noindent {\rm  (i)}\quad In case $p>2$  one has
$
 \|\nabla w\|_{L^\infty (B(0,r))}<\infty
 $
for any $r<1$.
 Moreover, if $w_{|\partial \ID}=0,$ then there is $C_p<\infty$ so that
 \[
 \|\nabla w\|_{L^\infty (\ID)}\leq C_p\|\Delta w\|_{L^p (\ID)}.
 \]

 \smallskip
  \noindent {\rm (ii)}\quad If  $w_{|\partial \ID}=0,$ and $1<p<2$, then $\|\nabla w\|_{2p/(2-p)} <\infty.$
 \end{lemma}
\begin{proof}
By the classical representation we have for $|z|<1$
\begin{equation}\label{e:POISSON1}w(z)
=\frac{1}{2\pi}\int_0^{2\pi}P(z,e^{i\varphi})w(e^{i\varphi})d\varphi+
\int_{\mathbf{U}}G(z,\omega)\Delta w(\omega)\,dA(\omega),
\end{equation}
where $P$ stands for the Poisson kernel and $G(z,\omega):=\frac{1}{2\pi}\log\big|\frac{1-z\overline{\omega}}{z-\omega}\big|$ for the Green's function of $\ID$. We observe first that since $G$ is real-valued, $|\nabla G|=2|\partial_z G|$ so that
\begin{equation}\label{nablaGreen}
|\nabla G(z,\omega)|=\frac{1}{2\pi}\big| \frac{-\overline{\omega}}{1-z\overline{\omega}}-\frac{1}{z-\omega}\big|\leq \frac{1}{\pi |z-\omega|}.
\end{equation}
Hence an application of H\"older's inequality shows that the second term in \eqref{e:POISSON1} has uniformly bounded gradient in $\ID
$. To conclude part (i) it suffices to to observe that the first term vanishes if $w_{|\partial \ID}=0$, and in the general case case it has uniformly bounded gradient in compact subsets of $\ID$. Finally, part (ii) follows immediately from \eqref{nablaGreen} by the standard mapping properties of the Riesz potential $I_1$ with the kernel $|z-\omega|^{-1}$, see \cite{stein1}.
\end{proof}

\begin{proof}[Proof of Theorem~\ref{th:first}]  It would be natural to try to generalise the ideas in \cite{JAM} where  differential inequalities were applied while treating related problems. However, it turns out that the approach of \cite{pisa}, where  the use of distance functions was initiated, is flexible enough for further development.

\emph{In the sequel we say $a \approx b$ if there is a constant $C\ge 1$ such that $a/C \le b \le C a$; and we say $a\lesssim b$ if there is a constant $C>0$ such that $a \le C b $.}

By our assumption on the domain, we may fix a diffeomorphism
$\psi:\overline{\Omega}\to\overline{\ID}$ that is $C^2$ up to the boundary.
 Denote $H:=1-|\psi|^2$, whence $H$ is  $C^2$-smooth in $\overline{\Omega}$  and vanishes on $\partial \Omega$ with $|\nabla H|\approx 1$ in a neighborhood of $\partial\Omega.$ We may then  define $h:\ID\to[0,1] $ by setting
\[
h(z):=H\circ w(z)= 1-|\psi(w(z))|^2\quad\textrm{for }\; z\in \ID.
\]
The quasiconformality of $f$ and the behavior of $\nabla H$ near $\partial\Omega$ imply that there is $r_0\in (0,1)$ so that the weak gradients satisfy
\begin{equation}\label{grad_equiv0}
|\nabla h(x)|\approx |\nabla  w(x)| \quad\textrm{for }\; r_0\leq |x|<1.
\end{equation}
Moreover, by Lemma \ref{le:interior_gradient}(i) we have
$
|\nabla h(x)|\lesssim |\nabla  w(x)|\leq C \quad\textrm{for }\;   |x|\leq r_0.
$
It follows that for any $q\in (1,\infty]$ we have that
\begin{equation}\label{grad_equiv1}
\nabla h\in L^q(\ID) \quad\textrm {if and only if }\; \nabla w\in L^q(\ID) .
\end{equation}

A direct computation, simplified by the fact that $H$ is real valued, yields that
\begin{equation}\label{laplace}
\Delta h= \Delta (H\circ w)) =(\Delta H)(w)(|w_z|^2+|w_{\overline{w}}|^2)+2{\rm Re}\Big(4H_{zz}(w)w_zw_{\overline{z}}+H_z(w)\Delta w\Big).
\end{equation}
Especially, since $H\in C^2(\overline{D})$  we have
\begin{equation}\label{laplace2}
|\Delta h|\lesssim   |\nabla w|^2+|g|.
\end{equation}

The higher integrability of quasiconformal self-maps of $\ID$ makes sure that $\nabla (\psi\circ w)\in L^q(\ID)$ for some $q>2$, which implies that  $ \nabla w\in L^q(\ID)$. By combining this with the fact that $g\in L^p(\ID)$ with $p>2,$ we deduce that $\Delta h\in L^{r}(\ID)$ with $r=\textrm{min}(p,q/2)>1.$ This information is not enough to us in case $q\leq4$,  but we will actually show that one may improve the situation to  $q>4$ via a bootstrapping argument based on the following observation: in our situation
\begin{equation}\label{boothstrap}
\textrm{if }\; \nabla w\in L^q(\ID)\;\; \textrm{with }\; 2<q<4,\quad \textrm{then actually }\;  \nabla w\in L^{2q/(4-q)}(\ID).
\end{equation}
In order to prove \eqref{boothstrap}, assume that $\nabla w\in L^q(\ID)$ for an exponent $q\in (2,4).$ Then \eqref{laplace2} and our assumption on $g$ verify that $\Delta h\in L^{q/2}(\ID).$ Since $h$ vanishes continuously on the boundary $\partial \ID$, we may apply Lemma \ref{le:interior_gradient}(ii) to obtain that $\nabla h\in L^{2q/(4-q)}(\ID)$ which yields the claim according to \eqref{grad_equiv1}.

We then claim that in our situation one has   $ \nabla w\in L^q(\ID)$ with some exponent $q>4$. For that end, fix an exponent $q_0>2$ obtained from the higher integrability of the quasiconformal map $w$ so that  $ \nabla w\in L^{q_0}(\ID)$. By diminishing $q_0$ if needed, we may well assume that
$q_0\in (2,4)$ and
$
q_0\not\in\{{2^n}/{(2^{n-1}-1)},\; n=3,4,\ldots\}.
$
Then we may iterate \eqref{boothstrap} and deduce inductively that  $\nabla w\in L^{q_k}(\ID)$ for $k=0,1,2\ldots k_0$, where
the indexes $q_k$ satisfy the recursion $q_{k+1}=\frac{2q_k}{4-q_k}$ and $k_0$ is the first index such that $q_{k_0}>4$. Such an index exists since by induction we have the relation $(1-2/q_{k+1})=2^k(1-2/q_0),$  for $k\geq 0.$

Thus we may assume that $ \nabla w\in L^q(\ID)$ with  $q>4$. At this stage  \eqref{laplace2} shows that $\Delta h\in L^{p\wedge (q/2)}(\ID).$ As  $p\wedge (q/2)>2,$   Lemma \ref{le:interior_gradient}(ii) verifies that $\nabla h$ is bounded. Finally,  by \eqref{grad_equiv1} we have the same conclusion for $\nabla w$, and hence $w$ is Lipschitz as claimed.

In order to verify the sharpness of the result, consider the following map
\[
w_0(z)=z\log^{a}\Big(\frac{e}{|z|^2}\Big),
\]
where $a\in (0,1/2)$ is fixed.
Then $w_0$ is a self-homeomorphism of $\ID$ that is quasiconformal with  continuous Beltrami-coefficient since  we may easily compute $(w_0)_z=\log^{a-1}\big(\frac{e}{|z|^2}\big)
\log\Big(\frac{e^{1-a}}{|z|^2}\Big)$ and $(w_0)_{\overline{z}}= -a\frac{\phantom{!}z\phantom{!} }{\overline{z }}\log^{a-1}\Big(\frac{e}{|z|^2}\Big)$ so that the complex dilatation of $w_0$ satisfies
\[
|\mu_{w_0}(z)|=\big|-a\frac{\phantom{!}z\phantom{!} }{\overline{z}}\Big(\log\Big(\frac{e^{1-a}}{|z|^2}\Big)\Big)^{-1} \big|\leq\frac{a}{1-a}<1.
\]
In addition, we see that $\Delta w_0\in L^2(\ID)$ since
\begin{eqnarray*}
|\Delta w_0(z)| &=&\big|4\frac{d}{d\overline{z}}(w_0)_z(z)\big|\;=\;\Big| \frac{4a}{\overline{z}}\log^{a-2}\Big(\frac{e}{|z|^2}\Big)\Big((a-1)-\log\Big(\frac{e}{|z|^2}\Big)\Big)\Big|\\
&\lesssim &|z|^{-1}\Big(\log\big(\frac{e}{|z|^2}\big)\Big)^{a-1}.
\end{eqnarray*}
Finally, it remains to observe that   $w$ is not Lipschitz at the origin.
\end{proof}

\begin{remark}\label{re1}
{\rm
If one invokes the known sharp $L^p$-integrability results of qc-maps (due to Astala, see \cite[Thm 13.2.3 ]{AIMb}) one sees that in the above proof no iteration is needed in case  $K<2$. One should also observe that the counterexample given above in the case $p=2$ is based already on the behaviour of $w$ near origin, not to any boundary effect, so in this sense    Theorem \ref{th:first} is quite sharp.
We have not seriously pursued the optimality question related to $C^2$-regularity assumption on $\Omega$. Here we simply  observe that  an easy modification of the proof  along the lines  of Section \ref{se:proof2} yields slightly stronger result, where instead of $C^2$-condition one only assumes that $\partial \Omega$ is $C^{1,\alpha}$-smooth for any $\alpha\in (0,1).$}
\end{remark}

\begin{remark}\label{re2}
{\rm
Assume that $w:B(0,1)\to B(0,1)$ is quasiconformal where $B(0,1)\subset\IR^d$ is the $d$-dimensional unit ball, $d\geq 3$ and such that $\Delta w_k\in L^p( B(0,1))$ with $p>n$  for each component of $w$ (here $k=1,\dots,d$). Then the above proof  applies with some modifications and shows that $w$ is Lipschitz. Actually, in a recent preprint \cite{kave} Astala and Manojlovi\' c  proved that quasiconformal  harmonic gradient mapping of the unit ball $B^3$ on to itself are bi-Lipschitz. They also  provide a short new proof of the Lipschitz-property of quasiconformal harmonic maps of the unit ball onto a domain with $C^{2}$ boundary on $\IR^d$ (c.f. \cite[Theorem~C]{JAM}). The results of \cite{kave} and of the present paper were obtained independently.}
\end{remark}

\section{Proof of Theorem \ref{th:second}: quantitative bounds as $\Omega\to\ID$}\label{se:proof2}

We start with an auxiliary lemma.

\begin{lemma}\label{le:toindentity} There exists a function $\psi:(1,2)\to \IR^+$  with the following property:
If  $w:\ID\to\ID$ is a $K$-quasiconformal self-map normalised with $\psi(0)=0,$  then
\[
\| \; |w_z|^2+|w_{\overline{z}}|^2-1\|_{L^3(\ID)}\; \leq \;\psi(K).
\]
Moteover, $\lim_{K\to 1^+}\psi (K)=0.$
\end{lemma}
\begin{proof} By the sharp area distortion $\|\nabla w\|_{L^6(\ID)}<\infty$ for $K<3/2$. By reflecting $w$ over the boundary $\partial \ID$ we may also assume that $w$ extends to a $K$-quasiconformal
map (still denoted by $w$) to the whole plane.  By rotation of needed, we may also impose the condition that $w(1)=1.$ Furthermore, we may even assume that $w_{\IC\setminus B(0,e^{3\pi})}$ is the identity map, since
we may use standard quasiconformal surgery (choose $k=(K-1)/(K+1)$ and $\alpha=2k$ in  \cite[Theorem 12.7.1]{AIMb}) to produce  $\frac{3K-1}{3-K}$-quasiconformal modification (still denoted by $w$) that equals to the original function $w$ in $\ID$ and  satisfies $w(z)=z$ for $|z|\geq e^{3\pi}.$ Especially, it is a principal solution. Since $\frac{3K-1}{3-K}\to 1$ as $K\to 1,$ and we are interested only on small values of $K$, it is thus enough to prove the corresponding claim for principal solutions with complex dilatation supported in $B(0,e^{3\pi}).$

Denote by $M$ the norm of the Beurling operator on $L^6(\IC).$ Fix $R_0>0$  and consider a principal solution $w$ to the Beltrami equation $w_{\overline{z}}=\mu w_z$ with
$|\mu |\leq k<1/2M$. Then we have the standard Neumann-series representation
\[
w_{\overline{z}}=\mu+\mu T\mu+\mu T\mu T\mu+\ldots \quad\textrm{and}\quad w_z-1=Tw_{\overline{z}}.
\]
We thus obtain that
\[
\|w_{\overline{z}}\|_{L^6(\IC)}\leq \|\mu \|_{L^6(\IC)}\left(1+ \frac{M}{2M}+ \left(\frac{M}{2M}\right)^2+\ldots\right)\leq 2\|\mu \|_{L^6(\IC)}
\leq Ck^{1/6}
\]
and, a fortiori $\|w_z-1\|_{L^6(\IC)}\leq MCk^{1/6}=C'k^{1/6}.$
We obtain the desired $L^3$-estimate for $|f_{\overline{z}}|^2$  since $k\to 0$ as $K\to 1.$ The estimate for $|f_{{z}}|^2-1$ follows by
noting that $\big| |f_z|^2-1\big|\leq |f_z-1|(|f_z-1|+2)$ and applying H\"older's inequality.
\end{proof}

Before proving the more general convergence result stated in the introduction it is useful to consider first the case where the image domain is fixed, and in fact equals $\ID$.

\begin{proposition}\label{pr:first} Assume that $p>2$. There exist a function \[
[1,\infty)\times [0,\infty) \ni (K,t)\to \widetilde C_p(K,t)
\] with the property: if $w:\ID\to\ID$ is a $K$-quasiconformal self map of the unit disc, normalised by $w(0)=0$, and with $\Delta w\in L^p(\ID),$ then one has
\[
\| \nabla w\|_{L^\infty(\ID)}\leq  \widetilde C_p(K,\|\Delta w\|_p).
\]
Moreover,  the function $\widetilde C_p$ satisfies
\begin{equation}\label{eq:goal}
\lim_{K\to1^+, \; t\to 0^+} \widetilde C_p(K,t)=1.
\end{equation}
\end{proposition}
\begin{proof}
We follow the line to the proof of Theorem \ref{th:first}, in particular we employ its notation, but this time we strive  to make the conclusion quantitative. We may well assume that $p\leq 3.$ Let us then
assume that $w$ is as in the assumption of the Proposition with $K<1+1/100,$ say.  In addition, we may freely assume that $w(1)=1.$ As the image domain is $\ID$, the function $h$  from the proof of Theorem \ref{th:first}, takes the form
\[
h(z)=1-|w(z)|^2.
\]
Let us write $h_0(z)=1-|z|^2,$ which corresponds to $h$ when  $w$ is the identity map. An application of \eqref{laplace}  and Lemma \ref{le:toindentity} allow us to estimate
\begin{eqnarray}\label{eq:laplacebound}
\|\Delta(h-h_0)\|_{L^p(\ID)} &=&\big\|4(1-|w_z|^2)-4|w_{\overline z}|^2+2{\rm Re}(\overline wg)\big\|_{L^p(\ID)}\nonumber\\
&\leq& 4\|(|w_z|^2-1)+|w_{\overline z}|^2\|_{L^3(\ID)}+\|g\|_{L^p(\ID)}\nonumber\\
&\leq& 4\psi(K)+\|g\|_{L^p(\ID)}.
\end{eqnarray}
Lemma \ref{le:interior_gradient} implies that
\begin{equation}\label{eq:gradientdifference}
\|\nabla h-\nabla  h_0\|\leq c_p(\psi(K)+\|g\|_{L^p(\ID)}).\nonumber
\end{equation}

The quasiconformality of $w$ implies  that  we have for almost every $z$ the estimate $|\nabla h (z)|\geq K^{-1}|(\nabla h_0)(w(z))| |\nabla w(z)|$. Since $|\nabla h_0(z)|=2|z|,$ we obtain by considering the annulus $1-\varepsilon \leq |z|<1$ with arbitrarily small $\varepsilon >0$  that
\begin{eqnarray}\label{eq:gradientdifference2}
\limsup_{|z|\to 1^-} |\nabla w(z)|&\leq& \frac{K}{2}\limsup_{|z|\to 1^-}\big(|\nabla h-\nabla  h_0|+|\nabla  h_0|\big)\nonumber\\
&\leq &
 \frac{c_pK}{2}\big(\psi(K)+\|g\|_{L^p(\ID)}\big) +K.
\end{eqnarray}

We now write $w$ in terms of the standard Poisson decomposition $w=u+f$, where $u$ is harmonic with $u_{|\partial\ID}=w_{|\partial\ID}$,
the term $f$ has vanishing boundary values and it satisfies $\Delta f=\Delta w=g$ in $\ID.$ Then maximum principle applies to the subharmonic function $|\nabla u|=|u_z|+|u_{\bar z}|=|a'|+|b'|$, where $a$ and $b$ are analytic functions such that $u=a+\overline{b}$, together with  Lemma~\ref{le:interior_gradient} shows that $|\nabla w|$ is bounded by $c\|g\|_{L^p(\ID)}.$ All, in all combing these observations with
\eqref{eq:gradientdifference2} we deduce that
\begin{eqnarray}\label{eq:gradientdifference3}
\sup_{|z|<1} |\nabla w(z)|&\leq&\limsup_{|z|\to 1^-}|\nabla u|+\sup_{|z|<1} |\nabla f(z)|\nonumber\\
&\leq &\limsup_{|z|\to 1^-}|\nabla w|+2\sup_{|z|<1} |\nabla f(z)|\nonumber\\
&\leq &
 \frac{c_pK}{2}\big(\psi(K)+\|g\|_{L^p(\ID)}\big) +K  +2c_p\|g\|_{L^p(\ID)}.
\end{eqnarray}
We may thus choose  for small enough $K$
$$
\widetilde C_p(K,t)= K+ \frac{c_pK}{2}\psi(K)+\frac{c_p(K+4)}{2}t,
$$
and the obtained bound has the desired behavior as $K\to 1$ and $t\to 0.$
\end{proof}

Below ${\rm Id}$ stands for the identity matrix ${\rm Id}:=\begin{pmatrix}1&0\\0&1\end{pmatrix}.$ We refer to \cite{stein1} for the standard definition of Sobolev spaces $W^{2,p}$ and for the H\"older(Zygmund)-classes $C^{\alpha}$ and $C^{1,\alpha}.$

\begin{definition}\label{def:controlled} Let $p>2$.
We say that the sequence of bounded Jordan domains $\Omega_n\subset\IC$, with  $0\in\Omega_n$ for each $n\geq 1$, converges in \emph{ $W^{2,p}$-controlled sense}  to the unit disc $\ID$ if  there exist sense-preserving diffeomorphisms $\Psi_n:\ID\to\Omega_n$, normalized by $\Psi_n(0)=0$,  such that
\begin{equation}\label{eq:controlled1}
\lim_{n\to\infty}\|D\Psi_n-{\rm Id}\|_{L^\infty (\ID)}=  0,\quad\textrm{and}\quad \|\Psi_n\|_{W^{2,p}(\ID)}\leq M_0\quad\textrm{for all }\; n\geq 1,
\end{equation}
where  $M_0<\infty$, and
\begin{equation}\label{eq:controlled2}
\|\Delta \Psi_n\|_{L^{p}(\ID)}\to 0\quad\textrm{as }\; n\to \infty.
\end{equation}
\end{definition}

\noindent One should observe that  since $\Psi_n \in W^{2,p}(D)$ with $p>2$ in the above definition,  it follows automatically that $\nabla \Psi_n\in C(\overline{\Omega})$. Hence asking $\Psi_n$ to be a diffeomorphism  makes perfect sense in terms and,  in particular, by \eqref {eq:controlled1} the map $\Psi_n$ is a bi-Lipschitz for large enough $n$.  Also, each  $\Omega_n$ is a  bounded $C^1$- Jordan domain in the plane. It turns out that the above condition is in a sense symmetric with respect to the domains $\ID$ and $\Omega :$
\begin{lemma}\label{le:inverse}
Assume that $\Omega_n$ tends to $\ID$ in a controlled sense and $(\Psi_n)$ is the associated sequence of diffeomorphisms satisfying the conditions of  definition {\rm\ref{def:controlled}}. Then the inverse maps $\Phi_n:=\Psi^{-1}_n:\Omega_n\to \ID$ satisfy
\begin{equation}\label{eq:controlled3}
\lim_{n\to\infty}\|D\Phi_n-{\rm Id}\|_{L^\infty (\Omega_n)}=  0,\quad\textrm{and}\quad \|\Phi_n\|_{W^{2,p}(\Omega_n)}\leq M'_0\quad\textrm{for all }\; n\geq 1,
\end{equation}
together with
\begin{equation}\label{eq:controlled4}
\|\Delta \Phi_n\|_{L^{p}(\Omega_n)}\to 0\quad\textrm{as }\; n\to \infty.
\end{equation}
\end{lemma}
\begin{proof}Conditions \eqref{eq:controlled3} follows easily by employing the formulas for the derivatives of the implicit function, after first approximating by smooth functions. Note, in regards to condition \eqref{eq:controlled4},  we note that in general the inverse of a harmonic diffeomorphism needs not to be harmonic,  so   \eqref{eq:controlled4} is not a direct consequence of  \eqref{eq:controlled2}. However, the first condition in   \eqref{eq:controlled1} tells us that the maximal complex dilatation $k_n$ of $\Psi_n$  tends to 0 as $n\to\infty$, so that $\Psi_n$ is asymptotically conformal and this makes  \eqref{eq:controlled4}  more plausible. Indeed, a direct computations shows that for $C^2$-diffeo $\Psi:\ID\to\Omega$ with maximal dilatation $k$  and controlled derivative $ |D\psi |, |(D\psi)^{-1}|\leq C$, it holds that
\[
\Delta \Phi =A\circ\Phi,
\]
where (recall that the Jacobian can be expressed as $J_\Psi=|\Psi_z|^2-|\Psi_\zbar |^2$)
\begin{eqnarray}\label{eq:iso}
A&=& \frac{4}{(J_\Psi)^3}\Bigg[ -\Psi_\zbar\bigg(\overline{\Psi_{z\zbar}}J_\Psi-\overline{\Psi_z}\Big(\overline{\Psi_z}\Psi_{zz}+\Psi_z\overline{\Psi_{z\zbar}}-\overline{\Psi_\zbar}\Psi_{z\zbar}-\Psi_\zbar\overline{\Psi_{\zbar\zbar}}\Big)\bigg)\nonumber
\\ &&+ \Psi_z\bigg(\overline{\Psi_{zz}}J_\Psi -\overline{\Psi_z}\Big(\overline{\Psi_z}\Psi_{z\zbar}+\Psi_z\overline{\Psi_{zz}}-\overline{\Psi_\zbar}\Psi_{\zbar\zbar}-\Psi_z\overline{\Psi_{z\zbar}}\Big)\bigg)\Bigg]
\end{eqnarray}
This formula   is obtained by  using as  a starting point the identity
$
\Delta\Phi =4({d}/{d\zbar})\Phi_z
$
and expressing $\Phi_z$ in a standard manner in terms of the derivatives of $\Psi$.
We next recall that $\Psi_z$ is bounded and $|\Psi_{\zbar}|\leq k\Psi_{z},$ and observe that in the right hand side  of \eqref{eq:iso} the terms that do not directly contain either $\Psi_{z\zbar}$ or $\Psi_{\zbar}$  as a factor sum up to
\[
\overline{\Psi_{zz}}(J_\Psi-|\Psi_z|^2)= -\overline{\Psi_{zz}}|\Psi_\zbar|^2,
\]
We obtain that
\[
|A|\; \lesssim\;  k|D^2\Psi |+|\Delta\Psi|,
\]
and  \eqref{eq:controlled4} follows by applying this on $\Psi_n.$
\end{proof}

We may now generalize Proposition \ref{pr:first} to include variable image domains that converge to the unit disc in controlled sense.

\begin{theorem}\label{th:second'}  Let $p>2$ and assume that  the planar Jordan domains $\Omega_n$ converge to $\ID$ in $W^{2,p}$-controlled sense.  Moreover, assume that $w_n:\ID\to\Omega_n $ is a $K_n$-quasiconformal normalised map normalised by $w(0)=0$, and with
\[
\lim_{n\to\infty} K_n=1 \quad\textrm{and}\quad  \lim_{n\to\infty} \|\Delta w_n\|_{L^p(\ID)}=0.
\]
Then for large enough $n$ the function $w_n$ is Lipschitz, and moreover its Lipschitz constant tends to 1 as $n\to\infty:$
\begin{equation}\label{eq:limit2}
\lim_{n\to\infty} \| \nabla w_n\|_{L^\infty(\ID)} =1.
\end{equation}
\end{theorem}
\begin{proof}
Let $\Psi_n: \ID\to \Omega_n$ be the maps as in definition \ref{def:controlled}. By renumbering, if needed, we may assume that that $|\Psi'_n(z)-1|<1/2$ for
all $n$ and $z\in \ID$. Write $\Phi_n=\Psi^{-1}_n$ and define
\[
\widetilde w_n:=\Psi^{-1}\circ w_n =\Phi_n\circ w_n:\ID\to\ID.
\]
 Then $\widetilde w_n$ is $\widetilde K_n$-quasiconformal, with $\widetilde K_n\to 1$ as $n\to \infty$ by the first condition in \eqref{eq:controlled3}.  Fix an index $q\in (2,p)$. By conditions  \eqref{eq:controlled1}, \eqref{eq:controlled3} and Proposition \ref{pr:first}, in  order to prove \eqref{eq:limit2}  we just need to verify  that
\begin{equation}\label{eq:enough}
\lim_{n\to\infty} \|\Delta \widetilde w_n\|_{L^q(\ID)} =0.
\end{equation}
A computation yields that
\begin{eqnarray}\label{eq:explicite_enough}
\Delta \widetilde w_n &=& (\Delta \Phi_n)(w_n)\Big( |(w_n)_z|^2+|(w_n)_\zbar|^2\Big)\nonumber\\
&&+\; 4\Big((\Phi_n)_{zz}(w)(w_n)_z(w_n)_\zbar+(\Phi_n)_{\zbar\zbar}(w_n)\overline{(w_n)_z(w_n)_\zbar}\Big)\nonumber\\
&&+\; \Big((\Phi_n)_z(w_n)\Delta w_n+(\Phi_n)_\zbar(w_n)\overline{\Delta w_n}\Big)\nonumber\\
&=:& S_1+S_2+S_3,
\end{eqnarray}
say.

Since $|D\Phi_n|$ remains uniformly bounded and we know that $ \|\Delta w_n\|_{L^p(\ID)}\to 0$, we see that $\|S_3\|_{L^p(\ID)}$ tends to zero as $n\to\infty$,  whence the same is true for the  $L^q$-norm. Set $\widetilde q:=\sqrt{qp}$ so that $q<\widetilde q<p.$
Since $\widetilde w_n$ is a normalized $K_n$-quasiconformal self-map of the unit disc $\ID$,  and $K_n\to 1$, we may assume (again by discarding small values of $n$ and relabeling, if needed)  by  the higher integrability of quasiconformal maps that $\int_{\ID}|\nabla w_n|^{2(\widetilde q/q)'} <C$  and $ \int_{\Omega}(J_{w_n^{-1}})^{(p/\widetilde q)'}dA(z) <C$ for all $n$. Here e.g.  $(\widetilde q/q)'$ stands for the dual exponent.  Denoting $k_n=(K_n-1)/(K_n+1)$ we thus obtain
for any measurable function $F$ on $\Omega$
\[\begin{split}
\int_{\ID}\big| F\circ w_n |(w_n)_z|^2
\big|^qdA(z)\; &\leq \;\Big( \int_{\ID}\big| F \circ w_n\big|^{\widetilde q} dA(z)\Big)^{q/\widetilde q}
\Big( \int_{\ID}| \nabla w_n|^{2(\widetilde q/q)'} dA(z)\Big)^{1/(\widetilde q/q)'}\\
&\lesssim \Big( \int_{\ID}\big| F \circ w_n\big|^{\widetilde q} dA(z)\Big)^{q/\widetilde q}\leq
\Big( \int_{\Omega}\big| F\big|^{\widetilde q} J_{w_n^{-1}}dA(z)\Big)^{q/\widetilde q}\\
&\lesssim
 \Big( \int_{\Omega}\big| F\big|^{ p} dA(z)\Big)^{q/p}\Big( \int_{\Omega}(J_{w_n^{-1}})^{(p/\widetilde q)'}dA(z)\Big)^{q/(\widetilde q(p/\widetilde q)')}
\\ &\leq \;  \Big( \int_{\Omega}\big| F\big|^{ p} dA(z)\Big)^{q/p}.
\end{split}
\]
By employing this formula and Lemma \ref{le:inverse} we obtain immediately that
\[
\| S_1\|_{L^q(\ID)}\lesssim \|\Delta \Phi_n\|_{L^p(\Omega)}\to 0\quad \textrm{as}\quad n\to\infty.
\]
In a similar vain
\[
\| S_2\|_{L^q(\ID)}\lesssim k_n\to 0\quad \textrm{as}\quad n\to\infty.
\]
This ends the proof of the theorem.
\end{proof}

We next examine what kind of convergence of the boundaries $\partial\Omega_n\to \partial \ID $  imply $W^{2,p}$-controlled convergence of the domains itself. First of all, given $\psi_n:\ID\to\Omega$ as in Definition \ref{def:controlled} we have $\Psi_n\in W^{2,p}(\ID),$ so by the standard trace theorem for  the Sobolev spaces  the induced map on the boundary satisfies ${\Psi_n}_{|\partial \ID}\in B^{2-1/p}_{p,p}(\ID).$ On the other hand,
for $p>2$ we may pick $\alpha, \alpha'\in (1/2,1)$ so that
\[
 C^{1,\alpha'}(\partial\ID)\subset B^{2-1/p}_{p,p}(\ID)\subset C^{1,\alpha}(\partial\ID),
\]
see \cite{Triebel}.
Hence about the best one can hope is to have a theorem where the boundary converges in $C^{1,\alpha}$ for some $\alpha >1/2.$ In fact, this can be realized:
\begin{theorem}\label{alpha}
Let $(\Omega_n)$ be a a sequence of bounded Jordan domains in $\IC$ such that  there is the parametrisation
\[
\partial \Omega_n=\{ f_n(\theta) \; | \; \theta \in (0,2\pi)\}.
\]
for each $n$, where  $f_n$ satisfies for some $\alpha >1/2$
\begin{equation}\label{eq:condition}
\| f_n(\theta)-e^{i\theta}\|_{L^\infty(\mathbf{T})}\to 0 \quad \textrm{as}\quad n\to\infty\quad \textrm{and}\quad  \sup_{n\geq 1}\| f_n(\theta)\|_{C^{1,\alpha}(\IT)}<\infty.
\end{equation}
Then the sequence $(\Omega_n)$ converges to $\ID$ in  $W^{2,p}$-controlled manner. In particular, the conclusion of Theorem \ref{th:second'} holds true for the sequence $(\Omega_n).$
\end{theorem}
\begin{proof} Let us first observe that instead of  \eqref{eq:condition} we may fix $\alpha'\in (1/2,\alpha)$ and assume that
\[
\| f_n(\theta)-e^{i\theta}\|_{C^{1,\alpha'}}\to 0 \quad \textrm{as}\quad n\to\infty .
\]
Namely, this follows by applying  interpolation on \eqref{eq:condition}.
Write $g_n(\theta)= f_n(\theta)-e^{i\theta}.$ By relabeling, if needed, we may assume that for all $n\geq 1$ we have  $\| g_n\|_{C^{1,\alpha}(\IT)}\leq 1/10,$ say. Since ${\rm Id}:\IT\to \IC$ is 1-bi-Lipschitz, and the Lipschitz norm of $g_n$ is small  we obtain that $f_n:\IT\to \partial\Omega_n$ is a diffeomorphism. We simply define $\Psi_n$ as the harmonic extension
\begin{eqnarray*}
\Psi_n(z) &= &\frac{1}{2\pi}\int_0^{2\pi}P(z,e^{it})f_n(e^{it})dt \;=\; z\;+\; \frac{1}{2\pi}\int_0^{2\pi}P(z,e^{it})g_n(e^{it})dt\\
&=& z\;+\; G_n(z),\qquad z\in \ID.
\end{eqnarray*}
Since $\|g'_n\|_\infty\to 0$ and $\| Hg'_n\|_\infty\to 0$ (recall that the Hilbert transform $H$ is continuous in $C^\alpha(\IT)$), we may also assume that
$|DG_n(z)|\leq 1/2$ for all $n$, and we have $\lim_{n\to\infty}\| DG_n\|_{L^\infty (D)} =0.$  Especially, $\Psi_n:\overline{\ID}\to\overline{\Omega_n}$ is $C^1$ and bi-Lipschitz, hence diffeomorphism. The first condition in \eqref{eq:controlled1} follows immediately, and condition \eqref{eq:controlled2} is immediate since $\Psi_n$ is harmonic. It remains to verify the second condition in  \eqref{eq:controlled1}. For that end observe that by \cite{stein1} the fact that
$\| g_n\|_{C^{1,\alpha}(\mathbf{T})}\leq C$ for all $n$ implies  that the Poisson extension satisfies
\[
\| D^2 G_n(z)\|\leq \frac{C'}{(1-|z|)^{1-\alpha}}.
\]
This obviously yields the desired uniform bound for $\|D^2 G_n\|_{L^p(\ID)}$ if we take $p<(1-\alpha)^{-1}.$
\end{proof}

Another condition is obtained by specializing to Riemann maps -- the proof of the the preceding theorem could also be  based on certain results of Smirnov concerning the regularity of conformal extensions and  the following lemma:

\begin{lemma}\label{def:controlled_riemann}  Let $p>2.$
The sequence of bounded Jordan domains $\Omega_n\subset\IC$ converges in $W^{1,p}$-controlled sense  to the unit disc $\ID$ if  the Riemann maps $F_n:\ID\to\Omega_n$  $($normalized by $F_n(0)=0$  and $\mathrm{arg}\,F'_n(0)>0)$ satisfy
\begin{equation}\label{eq:controlled}
\lim_{n\to\infty}\|F'_n-1\|_{L^\infty (\ID)}=  0,\quad\textrm{and}\quad \|F''_n\|_{L^p(D)}\leq M_0\quad\textrm{for all }\; n\geq 1,
\end{equation}
with some  $M_0<\infty.$
\end{lemma}
\begin{proof}
Obvious after the definition of controlled convergence.
\end{proof}

\begin{remark}\label{rem:section3} {\rm It is an open question whether one can weaken the condition   $\alpha >1/2$ in Theorem \ref{th:second}.}
\end{remark}

\section{Proof of Theorem \ref{th:smirnov} and Corollary \ref{co:smirnov} : A Smirnov theorem for qc-maps}\label{se:proof3}

\begin{proof}[Proof of Theorem~\ref{th:smirnov}] We first assume that $f$ is as in the theorem so that $\Delta f(z)\leq (1-|z|)^{-a}$ with $a\in (0,1).$
Then we are to show that the boundary map induced by $w$ is absolutely continuous.
For that end we need two simple lemmas.
\begin{lemma}\label{lemo}
Assume that $u\in C(\overline{D)}$ is a harmonic mapping of the unit disk into $\mathbf{C}$ such that the the $U:=u_{|\mathbf{T}}$ is a homeomorphism and $U(\mathbf{T})=\Gamma$ is a rectifiable Jordan curve.  Then   $|\Gamma_r|:=\int_{\mathbf{T}}|\partial_\theta u(re^{i\theta})|d\theta$ is increasing in $r$ so that $|\Gamma_r|\le |\Gamma|
$. Especially, the angular derivative of $u$ satisfies $\partial_\theta u(z)\in h^1$.
\end{lemma}
\begin{proof}[Proof of Lemma~\ref{lemo}]
By differentiating  the Fourier-series representation
$$
u(re^{i\theta})=\sum_{n=-\infty}^\infty\widehat g_n r^{|n|}e^{in\theta}
$$
we see immediately that $\partial_\theta u(z)$ is the harmonic
extension to $U$ of the distributional derivative $\partial_\theta g$.
By assumption, $g$ is of bounded variation, and hence $\partial_\theta g$ is a finite (signed) Radon measure, which implies that $\partial_\theta u\in h^1$. It is well-known (see \cite[11.17]{Rudin}) that for functions in $h^1$ the integral average $ \int_{\mathbf{T}}|\partial_\theta u(re^{i\theta})|d\theta$ is increasing in $r$.
\end{proof}
\begin{lemma}\label{wei}
Let $g\in L^p(\mathbf{U})$ with $p>1$. Then  there is a unique solution to the Poisson equation $\triangle v=g$ such that $v\in C(\overline{U})$ and $v_{|\ttt}=0.$ Moreover, the  weak derivative  $Dv$ can be modified in a set of measure zero so that
$$
\int_{0}^{2\pi}|Du(re^{i\theta})|d\theta \leq C(g)<\infty\qquad\textrm{for}\;\; r\in (1/2,1).
$$
\end{lemma}
\begin{proof} The classical regularity theory for elliptic equations (see \cite{ADN},\cite{gt}) yields a quick approach, as it guarantees that our Poisson equation has a unique solution  $v$ in the Sobolev space $W^{2,p}(U)$ (which is of course given by the Green potential, see \eqref{e:POISSON1}) and we have continuity up to the boundary. The derivatives
satisfy $\partial_z,\partial_{\overline z}\in W^{1,p}(U)$. Especially,
we then have $\|Dv\|_{W^{1,p}}(B(0,r))\leq C'$ for any $r\in (1/2,1)$. At this stage the trace theorem (see e.g. \cite{Triebel}) for the  space $W^{1,p}(U)$
and a simple scaling argument shows for a suitable representative
of $Dv$ that
$$
\|(Dv)_r\|_{B^{1-1/p}_{p,p}(\ID)}\leq C'\quad\textrm{for}\;\; r\in (1/2,1).
$$
Here $(Dv)_r$ stands for the function $\ttt\ni\theta\mapsto v(re^{i\theta}).$ The claim follows by observing the continous imbeddings
$B^{1-1/p}_{p,p}(\ID)\subset L^p(\ID)\subset L^1(\ID).$
\end{proof}

Recall also  that any
analytic (or anti-analytic) function in $h^1$ can be represented as the   Poisson integral of an $L^1$-function, see \cite[Theorem 17.11]{Rudin} or \cite{G}.
In order to proceed towards the absolute continuity of boundary values of $f$, write
$f=a+\overline{b}+v$, where $v$ solves $\Delta  v=g:=\Delta f$ with $v_{|\ttt}=0$ and $a$ and  $b$ are analytic in the unit disk.
Since $u:=a+\bar b=\mathcal{P}[f_{|\ttt}]$, where $f_{|\ttt}$ is a homeomorphism, it follows from Lemma~\ref{lemo} that $\partial_\theta u=i(za'-\overline{zb'})\in h^1(\mathbf{U})$, because $f(\mathbf{T})$ is a rectifiable curve.
Further, the weak derivatives satisfy
$$f_z=a'+v_z,\ \  f_{\bar z}=\overline{b'}+v_{\bar z}$$
Now we use that $$|f_{\bar z}|\le k|f_z|,\ \ \ k=\frac{K-1}{K+1}$$ which implies that $$|a'+v_z|\le k|b'+\overline{v_{\bar z}}|.$$
As $$b'=\frac{\bar z}{z} \overline{a'}-\frac{i}{z} \overline{u_\theta},$$ we obtain for $z\not=0$ that
\[\begin{split}|a'|&\le k\Big|\frac{\bar z}{z} \overline{a}'-\frac{i}{z}  \overline{u_\theta}+\overline{v_{\bar z}}\Big|+|v_z|.\end{split}\]
This yields for $|z|\geq 1/2$ the inequality, valid almost everywhere
\begin{eqnarray*}
|a'|\le \frac{1}{1-k} (2| \overline{u_\theta}|+|\overline{v_{\bar z}}|+|v_z|).
\end{eqnarray*}

Our assumption on the size of the Laplacian of $f$ yields that $\Delta f\in L^p(\ID)$ for some  $p>1$. By combining this with above inequality, and noting that   $\overline{u_\theta}\in h^1$ by Lemma \ref{lemo}, we infer   (using simple argument that uses Fubini as the above inequality holds only for a.e. $z$) that  $a'\in H^1$.  Then the relation  $b'=\frac{\bar z}{z} \overline{a'} -\frac{i}{z} \overline{u_\theta}$ verifies that also $b\in H^1$.
Thus $\partial_\theta u$ is the Poisson integral of an $L^1$ function, and we conclude that  $f_{|\ttt}=u_{|\ttt}$ is absolutely continuous.

\medskip

In order to prove the optimality of Theorem $\ref{th:smirnov}$, we are to construct quasiconformal maps with non-absolutely continuous boundary values, but at the same time with not too large Laplacian. For that end it is easier to  work in the upper half space $\IC^+:=\{ z: {\rm Im}z>0\}.$
We will produce the desired functions with the help of Zygmund measures.
Recall first that a {bounded and continuous} function  $g:\IR\to\IR$ is \emph{Zygmund} if
\[
\big| g(x+t)+g(x-t)-2g(x)\big|\leq C|t| \quad\textrm{for all}\;\; x,t\in\IR.
\]
The smallest possible $C$ above is the Zygmund norm of $g$. If $g$ is increasing, its derivative is a positive finite Borel measure, $g'=\mu$, on $\IR$ and we call $g$ a \emph{singular Zygmund function} if, in addition, $\mu$ is singular. It is well-known that there exists singular Zygmund measures, see \cite{pir} or \cite{kahane}. In general, we refer the reader to the interesting article \cite{aan} for further information on this type of measures.

We next recall a modified version of the Beurling-Ahlfors extension, due to Fefferman, Kenig and Pipher \cite{fkp}. For that end denote the Gaussian density by $\psi(x):=(2\pi)^{-1/2}e^{-x^2/2}$, and notice that $-\psi'(x)= -x\psi(x).$ As usual, for $t>0$ we define the dilation
$\psi_t(x):=t^{-1}\psi (x/t),$ and $\psi'_t$ is defined in analogous way. Then the extension $u$ of and (at most polynomially) increasing homeomorphism $g:\IR\to\IR$ is defined by setting
\begin{equation}\label{eq:fkp-extension}
u(x+it):= (\psi_t*g)(x)+i(-\psi'_t*g)(x),\quad \textrm{for all}\quad x+it\in\IC^+.
\end{equation}
Obviously, $u$ is smooth in $\IC^+$ and it has the right boundary values. We have:
\begin{lemma}[{\cite[Lemma 4.4.]{fkp}}]\label{le:fkp} If $g:\IR\to\IR$ is quasisymmetric, then the extension $u$ defined via \eqref{eq:fkp-extension} defines a quasiconformal homeomorphism of $\IC^+$ whose boundary map coincides with $g$.
\end{lemma}
We need one more auxiliary result:
\begin{lemma}\label{le:aux} Assume that $g:\IR\to\IR$ is Zygmund. Then the extension  \eqref{eq:fkp-extension} of $g$ satisfies
for alla $x\in\IR$ and $t>0$
\begin{eqnarray*}
|\Delta u(x+it)|&\leq& Ct^{-1},\qquad\textrm{and}\\
|\nabla u(x+it)|&\leq& C\max\big(1,\log(t^{-1})\big),
\end{eqnarray*}
where $C>0$ is a constant.
\end{lemma}
\begin{proof}
Let us first observe that if $g$ is Zygmund, then for any $\varphi\in W^{2,1}(\IR)$ (i.e. $\varphi,\varphi''\in L^1(\IR)$) we have
\begin{equation}\label{eq:zygmundestimate}
\left\|\frac{d^2}{dx^2}\varphi_t*g\right\|_{L^\infty(\IR)}=O(t^{-1}),\quad \textrm{for all }\; t>0.
\end{equation}
We note that this follows easily from the mere definition of Zygmund functions if $\varphi$ is even, but for  general $\varphi$ we shall use the  fact that  $g$ can be decomposed as the sum $g=\sum_{j=0}^\infty g_j,$ where $\|g_j\|_{L^\infty(\IR)} =O(2^{-j})$ and $\|g''_j\|_{L^\infty(\IR)} =O(2^{j})$ for all $j\geq 0,$ see \cite[Corollary 1, p. 256]{stein1}. We may compute in two ways
\[
\frac{d^2}{dx^2}(\varphi_t*g(x))\; =\; \int_{-\infty}^{\infty}\varphi_t(x-y)g''(y)dy\;=\;t^{-2}\int_{-\infty}^{\infty}\varphi''_t(x-y)g(y)dy.
\]
By assuming first that $t\leq 1$ with  $t\sim 2^{-k}$ we apply the first formula above to the sum $g=\sum_{j=0}^k g_j,$ and the second one to the remainder
$g=\sum_{j=k+1}^\infty g_j. $ By noting that $\int_{-\infty}^{\infty}|\varphi_t(y)|dy=O(1)$ and $\int_{-\infty}^{\infty}|\varphi''_t(y)|dy=O(1)$, we obtain
\[
|\frac{d^2}{dx^2}(\varphi_t*g(x))| = O\Big(\sum_{j=1}^k2^j + t^{-2}\sum_{j=k+1}^\infty2^{-j} \Big)\; = \; O(t^{-1}),
\]
which proves \eqref{eq:zygmundestimate} for $t\in (0,1]$. If $t>1$ we simply apply the second formula directly on the bound $\|g\|_{L^\infty(\IR)} <\infty$ and obtain $\|\frac{d^2}{dx^2}(\varphi_t*g)\|_{L^\infty(\IR)}\leq O(t^{-2})=O(t^{-1})$ for $t>1$.

We then consider the Laplacian of the extension $u$ of $g$. Since $\psi, \psi'\in W^{2,1}(\IR),$ we obtain immediately from \eqref{eq:zygmundestimate} that $|\frac{d^2}{dx^2}u(x+it)|=O(t^{-1})$ uniformly in $x\in\IR.$ In turn, to consider differentiation with respect to $t$,
assume that $\phi:\IR\to\IR$ is smooth and  $(1+|t|^2)\phi(t)$ is integrable. Then
\begin{eqnarray*}
\frac{d}{dt}\varphi_t*g(x) & =&  \int_{-\infty}^{\infty}\Big(-t^{-2}\varphi_t(x-y)-t^{-3}(x-y)\varphi'_t(x-y)\Big)g(y)dy\\ &=& \int_{-\infty}^{\infty}g(y)\frac{d}{dy}\Big(t^{-2}(x-y)\varphi_t(x-y))\Big)dy\\
&=& -t^{-1} \int_{-\infty}^{\infty}\frac{(x-y)}{t}\varphi \Big(\frac{x-y}{t}\Big) g'(y)dy\\
&=& (\varphi_1)_t*g'(x) ,
\end{eqnarray*}
where $\varphi_1(y):=-y\varphi (y).$ An iteration gives, by denoting $\varphi_2(y):=y^2\varphi (y),$
\begin{equation}\label{eq:reduction}
\frac{d^2}{dt^2}\Big(\varphi_t*g(x)\Big) =(\varphi_2)_t*g''(x) = \frac{d^2}{dx^2}\Big((\varphi_2)_t*g(x)\Big).
\end{equation}
Since all the functions $t\psi(t), t^2\psi(t),t\psi(t), t^2\psi(t)$ and their second derivatives are integral, we may apply  \eqref{eq:reduction} and obtain as before the desired estimate for  $\frac{d^2}{dt^2}u(x+it).$

The stated estimate for $\nabla u$ is proven in a similar way. We use the fact that for in the decomposition  $g=\sum_{j=0}^\infty g_j,$ one may in addition demand that $\|g_j'\|_\infty\leq C $ for all $j\geq 1$ (see \cite[Formula (53), p. 254]{stein2}), which yields as before for $t\sim 2^{-k}<1$
\[
\left|\frac{d}{dx}(\varphi_t*g(x))\right| = O\Big(\sum_{j=1}^k  1+ t^{-1}\cdot \sum_{j=k+1}^\infty2^{-j} \Big)\; = \; O\big(\log(t^{-1})\big).
\]
The case $t\geq 1$ is trivial, and the case of the $t$-derivative is reduced to estimating the $x$-derivative as before.
\end{proof}

After these preparations it is now a simple matter to produce the desired example. Let $g_0$ be a singular Zygmund function which is constant outside $[-1,1] $ so that   Set $g(x)=x+g_0(x)$ for $x\in\IR.$ As $g_0$ is Zygmund, the function $g$ is quasi symmetric. Then its Fefferman-Kenig-Pipher extension $u:\IC^+\to\IC^+$ is quasiconformal with non-absolutely continuous boundary values over $[-1,1].$ Since the extension of the linear function $x\mapsto x$ is linear, we see that the Laplacian of $u$ equal that of the extension of $g_0$, and by the previous lemma we obtain the estimate
$$
|\Delta u(x+it)|\leq Ct^{-1}\qquad \textrm{for all}\;\; x+it\in \IC^+.
$$

Next, let $h:\ID\to\Omega'$ be conformal, where $\Omega'$ is a bounded and smooth Jordan domain that is contained in the upper half space $\IC^+$
and contains $[-2,2]$ as a boundary segment. Denote $\Omega=u(\Omega')$ so that $\Omega$ is smooth by our construction. Finally, pick a conformal map $\widetilde h:\Omega\to \ID$ and define
$f:=u\circ h$. Function $f$ satisfies all the requirements since the main terms in the formula for the Laplacian of $f$ (compare to \eqref{eq:explicite_enough}) are  $|\Delta u |$ and  $|\nabla u |^2$, and the previous lemma also yields  suitable bounds also for the gradient term.
\end{proof}

\begin{proof}[Proof of Corollary~\ref{co:smirnov}] The example for optimality constructed above obviously works also for the Corollary. In a similar vain,  the proof of the positive direction of  Theorem \ref{th:smirnov} also applies as such for the Corollary since in the proof we  used as a starting point the fact that $\Delta u\in L^p(\ID)$ for some $p>1.$
\end{proof}

\begin{remark}\label{rem:smirnov1}{\rm
There exists singular Zygmund functions on the real line such that $g(x+t)+g(x-t)-2g(x)=o(t)$ with quantitative little $o$ in the right hand side -- the derivatives of such functions are sometimes called Kahane measures. A possible decay of  the right hand side is $o(t\log^{-1/2}(1/t))$ for small $t$, but
one cannot decrease the power of $\log$ here. Using this kind of measures in our construction gives examples
with Laplacian growth $o(t^{-1}),$ where the little $o$ can be made explicit.

However, it is an open problem   whether Corollary \ref{co:smirnov} is true  for the exponent $p=1,$ as merely implementing the Kahane measures  described above appears not to give enough extra decay for the Laplacian.
}\end{remark}

\end{document}